\documentclass[12pt]{amsart}
\usepackage{graphicx} 
\usepackage[english]{babel}
\usepackage{amssymb} 
\usepackage{amsmath}
\usepackage{mathtools}
\usepackage{amsthm}
\usepackage{mathrsfs}
\usepackage{enumerate}
\usepackage[    
backend=biber,
style=alphabetic,
sorting=nty
]{biblatex}
\usepackage[colorlinks,linkcolor=blue,citecolor=blue,urlcolor=blue, pdfcenterwindow, pdfstartview={XYZ null null 1.2}, pdffitwindow, pdfdisplaydoctitle=true]{hyperref}
\newtheorem{lemma}{Lemma}[section]

\newtheorem{theo}[lemma]{Theorem}

\newtheorem{rema}[lemma]{Remark}

\begin{document}

\newcommand{\N}{\mathbf N}
\newcommand{\Z}{\mathbf Z}
\newcommand{\R}{\mathbf R}
\newcommand{\Q}{\mathbf Q}
\newcommand{\C}{\mathbf C}

\title[Rapid Decay for Free Groups via Riesz--Thorin]
{A Riesz-Thorin Approach to the Rapid Decay Property for Free Groups}
\begin{abstract}
We establish $L^p$ bounds for operators associated with the quasi-regular representation of the free group on its Gromov boundary. The $p=2$ case recovers Haagerup's inequality, yielding a new interpolation-theoretic proof of the the Rapid Decay property for the free group.
\end{abstract}


\author{Guillaume Delord}
\email{guillaume.delord@etu.u-paris.fr}
\date{March 2026}

\maketitle

\tableofcontents

\section{Introduction}\label{Introduction}
It is well known that free groups satisfy the Rapid Decay property, also known as Haagerup's inequality, established by Haagerup in \cite{Haagerup1979}.

In this paper, we provide an alternative proof of this result by applying the Riesz–Thorin interpolation theorem to operators naturally associated with partial representations introduced later in Section~\ref{Notation}.

In particular, this permits us to carry out the computations entirely within the $L^1$ and $L^{\infty}$ settings, thereby avoiding the need to work directly in $L^2$.

The paper begins by introducing the necessary notation, then formulates an operator-theoretic condition ensuring the Rapid Decay property before presenting the proof.

It would be interesting to determine whether similar interpolation arguments may be applied to boundary representations of more general hyperbolic groups.

\subsection*{Acknowledgments}
The author would like to thank his PhD supervisor, Adrien Boyer, for his encouragement and guidance in the preparation of this work, and his Master's thesis supervisor, Christophe Pittet, for his generous support and for introducing him to the subject.

\section{Notation}\label{Notation}
Let $\Gamma$ be the free group of rank r with a fixed set of free generators $\{a_j\}_{j=1}^r$ and neutral element $e$.

Let $\mathcal{T}$ be the Cayley graph associated to right multiplication by generators or their inverses in $\Gamma$. It is a homogeneous tree of degree $2r$ with vertices $\Gamma$ and $\{x,y\}$ is an edge if and only if $y=xa_j$ or $x=ya_j$ for some $j$.

There is a natural graph metric on $\mathcal{T}$ assigning distance $1$ to adjacent vertices.

The distance from an element $g \in \Gamma$ to the identity element $e$ is called the length of $g$ and is denoted by $|g|$. $S_n$ denotes the sphere of elements of length $n$ from $e$.

Left multiplication in $\Gamma$ gives a left action on $\mathcal{T}$ which naturally makes $\Gamma$ a subgroup of $\operatorname{Aut}(\mathcal{T})$, the automorphism group of $\mathcal{T}$.
More precisely, the group $\Gamma$ acts on $\mathcal{T}$ with no inversions or rotations, but solely by translations  as it is widely discussed in the book \cite[Chap.~I, Section~3]{FigaTalamancaNebbia1991Trees}.

For convenience, we set $q:=2r-1$ which naturally appears throughout the computations, and thus $\mathcal{T}$ is a $q+1$ regular tree.
\newline

The space $\Omega$ is the boundary of the tree $\mathcal{T}$. It can be seen as the set of right-infinite products of the form  $a_{i_1}^{\epsilon_1}a_{i_2}^{\epsilon_2}\cdots a_{i_j}^{\epsilon_j} \cdots$ (where $1 \leq i_j \leq r$ and $\epsilon_j = \pm 1$ for all integer $j$) which are reduced ($i_j =i_{j+1}$ implies $\epsilon_j = \epsilon_{j+1}$). The left action of the group $\Gamma$ on $\mathcal{T}$ naturally extends to an action on $\Omega$.

The boundary $\Omega$ can also be seen as the set of half-geodesics starting at $e$. As explained in the introduction of \cite{KuhnSteger1992}, when equipped with the topology induced by the the maps from natural numbers to $\mathcal{T}$ (like geodesics), $\Omega$ is compact. A basis for this topology is given by the subsets $\Omega_{g}$ where $g$ is an element of $\Gamma$:
\begin{equation*}
    \Omega_{g}:=\{\omega \in \Omega : \omega_1\cdots \omega_{|g|} =g\}
\end{equation*}
where each $\omega_i$ belongs to the alphabet $\{a_j^{\pm 1}\}_{j=1}^r$.

Moreover, for any natural number $N$ the subsets $\{\Omega_{g}:|g|=N\}$ form a partition of $\Omega$.

The measure $\nu$ is the unique Borel regular probability measure on $\Omega$ invariant with respect to the maximal subgroup of elements of $\operatorname{Aut}(\mathcal{T})$ fixing $e$, such that
\begin{equation*}
    \nu(\Omega_{g})=\frac{1}{(q+1)q^{|g|-1}} \qquad (g \neq e ),
\end{equation*}
It is quasi-invariant under the action of $\Gamma$ on $\Omega$.

Let $\mathcal{H}$ denote the Hilbert space of square integrable complex-valued functions $L^2(\Omega,\nu)$.
Let also $\pi:\Gamma \to \mathcal{U}(\mathcal{H})$ be the unitary irreducible representation of $\Gamma$ defined by:
\begin{equation*}
    \pi(\gamma)\phi(\omega):= P(\gamma,\omega)^{\frac{1}{2}}\phi(\gamma^{-1}\omega)
\end{equation*}
where $P$ is the cocycle given by the Radon-Nikodym derivative:
\begin{equation*}
    P(\gamma,\omega) = \frac{d\gamma_*\nu}{d\nu}(\omega). 
\end{equation*}

It is well known (see, for instance, \cite[Chap.~3, Section~2]{FigaTalamancaNebbia1991}) that $P$ satisfies the following cocycle identities:
\begin{equation*}
    P(\gamma_1\gamma_2,\omega) =P(\gamma_2,\gamma_1^{-1}\omega)P(\gamma_1,\omega),\quad \gamma_1,\gamma_2 \in \Gamma, \omega \in \Omega.
\end{equation*}

 For $\gamma \in S_n$ and $i\in \{0,\cdots,n-1\}$, define 
  \begin{equation*}
      A_i(\gamma):=\{\omega \in \Omega:\omega_1\cdots\omega_i = \gamma_1\cdots \gamma_i \text{ and } \omega_{i+1}\neq \gamma_{i+1}\}=\{ \omega\in \Omega: (\omega,\gamma)_{e}=i\}.
  \end{equation*}
  where $(\cdot,\cdot)_e$ denotes the Gromov product with respect to the basepoint $e$ in the sense of \cite[Section~2.4]{Bourdon1995}.

For $\gamma \in S_n$, we also set
\begin{equation*}
    A_n(\gamma):=\Omega_\gamma.
\end{equation*}
  
Observe that $A_i(\gamma)$ is open as it can be expressed as the union of the sets $\Omega_{\gamma_1\cdots\gamma_i s}$, where $s$ ranges over generators of $\Gamma$ and their inverses, with $s\neq \gamma_{i+1}$ and $s\neq \gamma_i^{-1}$. In particular, the characteristic function $\mathbf{1}_{A_{i}(\gamma)}$ is continuous on the boundary $\Omega$.
\newline

Now, consider the commutative $C^*$-algebra $C(\Omega)$ of complex-valued continuous functions on the boundary equipped with the supremum norm and its group of *-automorphisms $\operatorname{Aut}(C(\Omega))$.

We have a group morphism $\alpha: \Gamma\to \operatorname{Aut}(C(\Omega))$ defined by:
\begin{equation*}
    \alpha(\gamma)F(\omega):=F(\gamma^{-1}\omega), \quad F \in C(\Omega), \gamma \in \Gamma, \omega \in \Omega,
\end{equation*} 
which is automatically continuous since $\Gamma$ is discrete,
and a $C^*$-representation $m:C(\Omega) \to \mathcal{B}(\mathcal{H})$ given by multiplication operator on $\mathcal{H}$, that is:
\begin{equation*}
    m(F)\phi := F\phi, \quad F \in C(\Omega), \phi \in \mathcal{H}.
\end{equation*} 
We then have:
\begin{align*}
    \pi(\gamma)m(F)\pi(\gamma)^*\phi(\omega) &= P(\gamma,\omega)^{\frac{1}{2}}\big(F\pi(\gamma^{-1})\phi\big)(\gamma^{-1}\omega)\\
    &= P(\gamma,\omega)^{\frac{1}{2}}F(\gamma^{-1}\omega)P(\gamma^{-1},\gamma^{-1}\omega)^{\frac{1}{2}}\phi(\gamma\gamma^{-1}\omega)
\end{align*}
Using the cocycle identity for $P$, we obtain the covariance relations:
\begin{equation}\label{covariance relations}
    \pi(\gamma)m(F)\pi(\gamma)^* = m(\alpha(\gamma)F) \quad F \in C(\Omega), \gamma \in \Gamma
\end{equation}
In other words, the triple $(m,\pi,\mathcal{H})$ constitutes a covariant representation of the $C^*$-dynamical system $(C(\Omega),\Gamma,\alpha)$ as described in \cite[Section~7.4.1 and Section~7.4.8]{Pedersen2018Cstar}.

For an element $\gamma$ of length $n$ and $i\in \{0,\cdots,n\}$, the corresponding partial representations are defined by:
\begin{equation*}
\pi_{i}(\gamma):=m(\mathbf{1}_{A_{i}(\gamma)})\pi(\gamma).
\end{equation*}
Replacing $\gamma$ by $\gamma^{-1}$ and taking $F=\mathbf{1}_{A_{i}(\gamma)}$ in (\ref{covariance relations}), we obtain:
\begin{equation*}
    \pi_{i}(\gamma)=\pi(\gamma)m(\alpha(\gamma^{-1})\mathbf{1}_{A_{i}(\gamma)})=\pi(\gamma)m(\mathbf{1}_{A_{n-i}(\gamma^{-1})})
\end{equation*}
where the last equality is a consequence of the following lemma.
\begin{lemma}
    For all $i$ in $\{0,\cdots,n\}$ and $\gamma$ in $S_n$,
    \begin{equation*}
        A_{n-i}(\gamma^{-1})= \gamma^{-1}A_i(\gamma)
    \end{equation*}
\end{lemma}
\begin{proof}
    Let $\omega \in A_{n-i}(\gamma^{-1})$. This means
    \begin{equation*}
        \omega_1 \cdots \omega_{n-i}= \gamma_n^{-1}\cdots\gamma_{i+1}^{-1} \text{ and } \omega_{n-i+1}\neq \gamma_{i}^{-1}
    \end{equation*}
    In particular
    \begin{equation*}
        \omega = \gamma^{-1}(\gamma\omega)=\gamma^{-1}(\gamma_1\cdots\gamma_i\omega_{n-i+1}\cdots) \in \gamma^{-1}A_i(\gamma)
    \end{equation*}
    and we have $A_{n-i}(\gamma^{-1})\subseteq \gamma^{-1}A_i(\gamma)$.

    Applying the same argument with $i$ replaced by $n-i\in \{0,\cdots,n\}$ and $\gamma$ by $\gamma^{-1}\in S_n$ yields $A_{i}(\gamma)\subseteq \gamma A_{n-i}(\gamma^{-1})$ which is equivalent to the reverse inclusion.
\end{proof}

In this form, the adjoint of $\pi_i(\gamma)$ can be computed explicitly:
\begin{equation}\label{adjoint}
    \pi_i(\gamma)^*=m(\mathbf{1}_{A_{n-i}(\gamma^{-1})})^*\pi(\gamma)^*
    =m(\overline{\mathbf{1}_{A_{n-i}(\gamma^{-1})}})\pi(\gamma^{-1})
    =\pi_{n-i}(\gamma^{-1}).
\end{equation}
To simplify notation, we omit $m$ and write:
\begin{equation*}
    \pi_{i}(\gamma)=\mathbf{1}_{A_{i}(\gamma)}\pi(\gamma).
\end{equation*}
\section{The Rapid Decay property for the free group}

In this section, we establish the Rapid Decay Property for the free group, formulated in terms of operator-norm bounds.

Given two measure spaces $X$ and $Y$, which will always be clear from the context, we write $\|T\|_{s\to s'}$ for the norm of a bounded linear operator $T: L^s(X) \to L^{s'}(Y)$, with $1\leq s,s'\leq +\infty$. Namely,
\begin{equation*}
    \|T\|_{s\to s'}:=\underset{\|v\|_{L^s(X)}\leq 1}{\sup}\|Tv\|_{L^{s'}(Y)}
\end{equation*}

Using the above notation, the property under consideration may be stated as:

When $f$ is an element of $\mathbb{C}\Gamma$ with support in $S_n$, the following estimate holds:
\begin{equation}\tag{RD}\label{RD}
    \|\pi(f)\|_{2\to 2}\leq (n+1)\|f\|_{l^2(\Gamma)}.
\end{equation}

\subsection{An Operator Condition Ensuring Rapid Decay}\label{Section sufficient RD}
We first decompose $\pi(f)=\overset{n}{\underset{i=0}{\sum}}\pi_i(f)$,
where each $\pi_i(f)$ is defined by $\underset{\gamma \in S_n}{\sum}f(\gamma)\pi_i(\gamma)$, and observe that it is enough to prove that each $\pi_i(f)$ is bounded by $\|f\|_{l^2(\Gamma)}$.

A straightforward application of the Cauchy-Schwarz inequality yields:
\begin{equation*}
    |\langle\pi_i(f)\phi,\psi\rangle|
    \leq
    \underset{\gamma \in S_n}{\sum}|f(\gamma)|.|\langle\pi_i(\gamma)\phi,\psi\rangle|
    \leq
    \|f\|_{l^2(\Gamma)}\big(\underset{\gamma \in S_n}{\sum}|\langle\pi_i(\gamma)\phi,\psi\rangle|^2 \big)^{\frac{1}{2}}.
\end{equation*}
The summand on the right-hand side may be rewritten as:
\begin{align*}
    &|\langle\pi_i(\gamma)\phi,\psi\rangle|^2
    =\langle\pi_i(\gamma)\phi,\psi\rangle \overline{\langle\pi_{i}(\gamma)\phi,\psi\rangle}
    =\langle\pi_i(\gamma)\phi,\psi\rangle \langle\psi,\pi_{i}(\gamma)\phi\rangle\\
    &=\langle\pi_i(\gamma)\phi,\psi\rangle \langle\pi_{n-i}(\gamma^{-1})\psi,\phi\rangle
    =\langle\pi_i(\gamma)\otimes\pi_{n-i}(\gamma^{-1})\phi\otimes\psi,\psi\otimes\phi\rangle
\end{align*}
 where the third equality holds because $\pi_i(\gamma)^*\overset{(\ref{adjoint})}{=}\pi_{n-i}(\gamma^{-1})$.
 
Therefore, it is natural to consider the bounded linear operator acting on the algebraic tensor product $L^{2}(\Omega,\nu)\otimes L^{2}(\Omega,\nu)$:
\begin{equation*}
    T_{n,i}:=\underset{\gamma \in S_n}{\sum}\pi_i(\gamma)\otimes\pi_{n-i}(\gamma^{-1})
\end{equation*}
In this setting, $\langle T_{n,i}\phi\otimes\psi,\psi\otimes\phi\rangle = \underset{\gamma \in S_n}{\sum}|\langle\pi_i(\gamma)\phi,\psi\rangle|^2$ is a positive real number and, since $\pi(f)=\overset{n}{\underset{i=0}{\sum}}\pi_i(f)$, it suffices to prove the following condition to obtain (\ref{RD}):
\begin{equation*}
    \underset{\|\phi\|_2,\|\psi\|_2\leq 1}{\sup}\langle T_{n,i}\phi\otimes\psi,\psi\otimes\phi\rangle \leq 1 \quad (0\leq i \leq n)
\end{equation*}

\subsection{A proof based on the Riesz–Thorin Theorem}
Since $P(\gamma,\cdot)$ takes only finitely many values when $\gamma\in S_n$, the formula:
\begin{equation*}
    \pi(\gamma)\otimes \pi(\gamma^{-1})F(\omega,\omega') = P(\gamma,\omega)^{\frac{1}{2}}P(\gamma^{-1},\omega')^{\frac{1}{2}}F(\gamma^{-1}\omega,\gamma\omega')
\end{equation*}
defines a bounded linear operator on 
$L^s(\Omega \times \Omega, \nu \otimes \nu)$ for all $s \in [1,\infty]$, 
which is unitary when $s = 2$.

It follows that $T_{n,i}$ can be regarded as a bounded linear operator on $L^s(\Omega \times \Omega, \nu \otimes \nu)$ which naturally extends the previous definition on pure tensors, for $s = 2$.
\newline

The Riesz-Thorin theorem will play a central role in our argument. The proof can be found in \cite[Chapter~IX, Section~4]{ReedSimon1975II} or \cite[Thm~21]{tao245c}.
\begin{theo} (Riesz-Thorin from \cite[Thm~21]{tao245c})
    Let $X$ and $Y$ be two measure spaces.
    
    Let $1\leq p_0,p_1,q_0,q_1\leq \infty$ and $T:L^{p_0}(X)+L^{p_1}(X) \to L^{q_0}(Y)+L^{q_1}(Y)$ such that $\|T\|_{p_0\to q_0}\leq B_0$ and $\|T\|_{p_1\to q_1} \leq B_1$. Then we have:
    \begin{equation*}
    \|T\|_{p_{\theta}\to q_{\theta}}\leq B_0^{1-\theta}B_1^{\theta}        
    \end{equation*}
    for all $0<\theta<1$ where $\frac{1}{p_{\theta}}=\frac{1-\theta}{p_0}+ \frac{\theta}{p_1}$ and $\frac{1}{q_{\theta}}=\frac{1-\theta}{q_0}+ \frac{\theta}{q_1}$.
\end{theo}

The strategy is to show that $T_{n,i}$ satisfies the hypotheses of the theorem for $p_0=q_0=\infty$, $p_1=q_1=1$ and $B_0=B_1=1$.

The following result concerns a geometric quantity that arises naturally in the computation of $\|T_{n,i}\|_{\infty\to\infty}$ and $\|T_{n,i}\|_{1\to 1}$:

\begin{lemma}\label{geometric lemma}
Let $i,n$ be natural numbers such that $i\leq n$. Then we have:
    \begin{equation*}
        \underset{\omega,\omega'\in\Omega}{\sup}\underset{\gamma \in S_n}{\sum}\mathbf{1}_{A_i(\gamma)}(\omega)\mathbf{1}_{A_{n-i}(\gamma^{-1})}(\omega')=1
    \end{equation*}
\end{lemma}
\begin{proof}[Proof of Lemma 3.2]
Let $\omega, \omega'$ be two elements of $\Omega$, and assume the existence of $\gamma \in S_n$ satisfying $(\omega, \omega')\in A_i(\gamma)\times A_{n-i}(\gamma^{-1})$. Equivalently, we have:
    $$
    \begin{cases}
        \gamma_1\cdots\gamma_i  = \omega_1\cdots \omega_i,  & \gamma_{i+1} \neq \omega_{i+1} \\
         \gamma^{-1}_n\cdots\gamma^{-1}_{i+1}  = \omega'_1\cdots \omega'_{n-i},  & \gamma^{-1}_{i} \neq \omega'_{n-i+1}
    \end{cases}
    $$
    Consequently, $\gamma=\omega_1\cdots\omega_i(\omega'_{n-i})^{-1}\cdots(\omega'_1)^{-1}$ is uniquely determined, from which we deduce that:
    \begin{equation*}
        \underset{\gamma \in S_n}{\sum}\mathbf{1}_{A_i(\gamma)}(\omega)\mathbf{1}_{A_{n-i}(\gamma^{-1})}(\omega') \leq 1.
    \end{equation*}
    More precisely, we distinguish cases according to the pair $(\omega, \omega')$.
    The formula $\omega_1\cdots\omega_i(\omega'_{n-i})^{-1}\cdots(\omega'_1)^{-1}$ gives a word of length $n=|\gamma|$ precisely when no cancellation occurs at the concatenation point and we obtain:
    $$
    \underset{\gamma \in S_n}{\sum}\mathbf{1}_{A_i(\gamma)}(\omega)\mathbf{1}_{A_{n-i}(\gamma^{-1})}(\omega')
    =\begin{cases}
        1 & (\text{ if } \omega_{i+1}^{-1}\neq \omega'_{n-i}\neq \omega_i \neq (\omega'_{n-i+1})^{-1})\\
        0 & (\text{otherwise)}
    \end{cases}$$
    Therefore we conclude that $\underset{\omega,\omega'\in\Omega}{\sup}\underset{\gamma \in S_n}{\sum}\mathbf{1}_{A_i(\gamma)}(\omega)\mathbf{1}_{A_{n-i}(\gamma^{-1})}(\omega')$ is exactly equal to $1$.    
\end{proof}
\begin{rema}
    This lemma is closely related to \cite[Prop~6.1]{Boyer2021Spherical}. Indeed, its proof relies on counting the elements of the sets:
    \begin{equation*}
        \{\gamma \in S_n : (\omega,\omega') \in A_i(\gamma)\times A_{n-i}(\gamma^{-1})\}
    \end{equation*}
    More precisely, one has:
    \begin{equation*}
        \underset{\gamma \in S_n}{\sum}\mathbf{1}_{A_{n-i}(\gamma^{-1})}(\omega)\mathbf{1}_{A_i(\gamma)}(\omega') = \#\{\gamma \in S_n : (\omega,\omega') \in A_i(\gamma)\times A_{n-i}(\gamma^{-1})\}
    \end{equation*} 
\end{rema}
We now deduce the following lemma.
\begin{lemma}
For all natural numbers $i$ and $n$ such that $i\leq n$, we have:
\begin{equation*}    \|T_{n,i}\|_{\infty \to\infty}\leq 1 \text{ and } \|T_{n,i}\|_{1\to 1}\leq 1
\end{equation*}
\end{lemma}
\begin{proof}
    It is well known that, if $\gamma$ is in $S_n$ and $\omega$ is in $A_i(\gamma)$, then $P(\gamma,\omega) = q^{2i - n}$; see, for instance, \cite[Lemma~2.10]{BoyerLobos2017}.
    Then, for all $(\omega, \omega') \in A_i(\gamma) \times A_{n-i}(\gamma^{-1})$, we have $P(\gamma, \omega) P(\gamma^{-1}, \omega') = q^{2i-n}q^{2(n-i)-n}=1$ and we obtain:
    \begin{equation*}
        T_{n,i}F(\omega,\omega')=\underset{\gamma \in S_n}{\sum}\mathbf{1}_{A_i(\gamma)}(\omega)\mathbf{1}_{A_{n-i}(\gamma^{-1})}(\omega')F(\gamma^{-1}\omega,\gamma\omega').
    \end{equation*}
    This yields $\|T_{n,i}\|_{\infty\to\infty}\leq B_0:=\underset{\omega,\omega'\in\Omega}{\sup}\underset{\gamma \in S_n}{\sum}\mathbf{1}_{A_i(\gamma)}(\omega)\mathbf{1}_{A_{n-i}(\gamma^{-1})}(\omega')\overset{ (\ref{geometric lemma})}{=}1$.
    
   Furthermore, using Fubini's theorem twice, performing the changes of variables $\gamma^{-1}\omega \mapsto \omega$ and $\gamma\omega' \mapsto \omega'$, and exploiting the cocycle property of $P$, we have:
   \begin{equation*}
       \underset{\Omega\times\Omega}{\int}|T_{n,i}F|d\nu\otimes\nu
       = \underset{\Omega\times\Omega}{\int}\big(\underset{\gamma \in S_n}{\sum}\mathbf{1}_{A_{n-i}(\gamma^{-1})}(\omega)\mathbf{1}_{A_i(\gamma)}(\omega')\big)|F(\omega,\omega')|d\nu\otimes\nu(\omega,\omega').
   \end{equation*}
    which leads to $\|T_{n,i}\|_{1\to 1}\leq B_1:=\underset{\omega,\omega'\in\Omega}{\sup}\underset{\gamma \in S_n}{\sum}\mathbf{1}_{A_{n-i}(\gamma^{-1})}(\omega)\mathbf{1}_{A_i(\gamma)}(\omega')=B_0=1$.
\end{proof}
The proof is completed by combining the previous lemma with the Riesz--Thorin theorem for $\theta = \frac{1}{2}$, which yields the sufficient condition for (\ref{RD}) as seen in section \ref{Section sufficient RD}:
\begin{equation*}
    \underset{\|\phi\|_2,\|\psi\|_2\leq 1}{\sup}\langle T_{n,i}\phi\otimes\psi,\psi\otimes\phi\rangle\leq \|T_{n,i}\|_{2\to 2}\leq B_0^{\frac{1}{2}}B_1^{\frac{1}{2}} = 1.
\end{equation*}

\begin{rema}
   The argument in fact yields uniform $L^p$-bounds for the operators $T_{n,i}$.
More precisely, combining the previous lemma with the Riesz–Thorin interpolation theorem, we obtain that for every $p\in[1,\infty]$,
    \begin{equation*}
        \|T_{n,i}\|_{p\to p}\leq B_0^{1-\frac{1}{p}}B_1^{\frac{1}{p}} = 1
    \end{equation*}
\end{rema}

\printbibliography

\end{document}